\documentclass[11pt,letterpaper,reqno]{amsart}
\usepackage[includehead,includefoot,margin=24mm]{geometry}
\usepackage{amsfonts}
\usepackage{amsmath}
\usepackage{amssymb}
\usepackage{amsthm}
\usepackage{times}
\usepackage{color}
\usepackage{enumerate}
\usepackage[T1]{fontenc}
\usepackage[utf8]{inputenc}
\usepackage{graphicx}
\usepackage{nicefrac}
\usepackage{xcolor}
\usepackage[all]{xy}
\usepackage{comment}

\usepackage{float}
\usepackage[normalem]{ulem}
\usepackage{pgfplots}
\pgfplotsset{compat=1.15}
\usepackage{mathrsfs}
\usetikzlibrary{arrows}

\usepackage{thmtools}
\usepackage{hyperref}
\hypersetup{
    colorlinks=true,
    linkcolor=blue,
    citecolor=blue,
    filecolor=blue,
    urlcolor=blue
}
\usepackage[nameinlink,capitalize,noabbrev]{cleveref}

\declaretheorem[name=Theorem,numberwithin=section]{theorem}
\declaretheorem[name=Theorem,numbered=no]{theorem*}
\declaretheorem[sibling=theorem,name=Lemma]{lemma}
\declaretheorem[sibling=theorem,name=Proposition]{proposition}

\declaretheorem[sibling=theorem,name=Example,style=definition]{example}
\declaretheorem[sibling=theorem,name=Remark,style=remark]{remark}

\newcommand{\GL}{{\sf{GL}}\,}

\newcommand{\XX}{\mathbf{x}}

\newcommand{\KK}[0]{\ensuremath{\mathbf{k}}}

\newcommand{\CC}[0]{\ensuremath{\mathbb{C}}}
\newcommand{\ZZ}[0]{\ensuremath{\mathbb{Z}}}

\newcommand{\GM}[0]{\ensuremath{\mathbb{G}_{\mathrm{m}}}}

\newcommand{\supp}[0]{\ensuremath{\operatorname{supp}}}

\newcommand{\Aut}[0]{\ensuremath{\operatorname{Aut}}}

\newcommand{\homo}[0]{\ensuremath{\operatorname{Hom}}}

\newcommand{\Der}[0]{\ensuremath{\operatorname{Der}}}

\makeatletter
\newcommand*\bigcdot{\mathpalette\bigcdot@{.5}}
\newcommand*\bigcdot@[2]{\mathbin{\vcenter{\hbox{\scalebox{#2}{$\m@th#1\bullet$}}}}}
\makeatother

\begin{document}

	%\setlength{\baselineskip}{0.54cm}        % Previous 0.56
	%
	%%%%%%%%%%%%%%%%%%%%%%%%%%%%%%%%%%%%%%%%%%%%%%%%%%%%%%%%%%%%%%%%%%
	%
\title[Christophersen's problem for monomial algebras]{Christophersen's problem for monomial algebras}
	
\date{\today}

\thanks{{\it 2020 Mathematics Subject 
  Classification}: 13F55, 13N15, 13E10, 14L30.\\
 \mbox{\hspace{11pt}}{\it Key words}: Christophersen's problem, monomial algebras, Demazure roots, automorphism groups.\\
\mbox{\hspace{11pt}} The first author was partially supported by ANID via Proyecto Fondecyt Iniciaci\'on N\textsuperscript{o}11260788. The second author was partially supported by ANID via Proyecto Fondecyt Regular N\textsuperscript{o}1240101 and Proyecto Fondecyt Exploraci\'on N\textsuperscript{o}13250049.}

\author{Roberto D\'iaz}
\address{Departamento de Matemáticas, Facultad de Ciencias, Universidad de La Serena, Av. Juan Cisternas 1200, La Serena, Chile.
}%
\email{roberto.diazv1@userena.cl}

\author{Alvaro Liendo}
\address{Instituto de Matem\'atica y F\'isica, Universidad de Talca,
 Casilla 721, Talca, Chile.}%
\email{aliendo@utalca.cl}

\begin{abstract}
Christophersen's problem predicts that the connected component of the automorphism group of a finite-dimensional local algebra $A$ of dimension $\ell$ over an algebraically closed field of characteristic zero has dimension at least $\ell-1$, with equality if and only if $A$ is isomorphic to $\KK[t]/(t^{\ell})$. We settle this problem in the class of monomial algebras. Using the combinatorial structure of the irredundant irreducible decomposition of a monomial ideal, we prove the predicted inequality and characterize the equality case. We further obtain stronger lower bounds depending on whether the defining monomial ideal is reducible or irreducible, and we determine all monomial algebras attaining each of these bounds.
\end{abstract}	

\maketitle

\section*{Introduction}

Let $\KK$ be an algebraically closed field of characteristic zero, and let $A$ be a finite-dimensional local $\KK$-algebra of dimension $\ell$, with maximal ideal $\mathfrak{m}$. The automorphism group $\Aut_\KK(A)$ is naturally a linear algebraic group, whose Lie algebra is the algebra of derivations $\Der(A)$; see~\cite[II,~\textsection 4,~2.3]{DG70}. The structure of $\Aut^0_\KK(A)$ and of $\Der(A)$ encodes essential infinitesimal information about $A$, and has been studied from several perspectives.

The dimension of $\Aut^0_\KK(A)$ has a concrete geometric meaning in the deformation theory of singularities, where finite-dimensional local algebras arise as \emph{fat points}. In their study of the versal deformations of cyclic quotient singularities and of the obstruction spaces of rational surface singularities, Christophersen~\cite{Chr91} and Behnke and Christophersen~\cite{BeCh91} systematically related the deformation theory of a singularity to that of its generic hypersurface sections, descending from rational surface singularities through curve singularities to fat points. In this setting, Wahl's conjecture on the dimension of smoothing components, proven by Greuel and Looijenga~\cite{GL85}, together with an observation of Mond and van Straten recorded in~\cite[Section~6.2]{BeCh91}, shows that over $\CC$ each smoothing component of $\CC\{x_1,\ldots,x_r\}/\mathfrak{m}^2$ has dimension $r^2 = \dim \Der(\CC\{x_1,\ldots,x_r\}/\mathfrak{m}^2)$. The same argument gives that \emph{every} smoothing component of the versal base space of a smoothable fat point with local algebra $A$ has dimension exactly $\dim \Der(A) = \dim \Aut^0_\CC(A)$. From an algebraic viewpoint, qualitative properties of $\Aut^0_\KK(A)$ have been studied extensively, most notably its solvability, starting from Yau's theorem on derivations of isolated hypersurface singularities~\cite{Yau91} and continuing with the work of Schulze~\cite{Sch10} and Perepechko~\cite{Per14}.

A natural question in this direction is to determine the minimal possible dimension of $\Aut^0_\KK(A)$. The truncated polynomial ring $\KK[t]/(t^{\ell})$ achieves $\dim \Aut^0_\KK(\KK[t]/(t^{\ell})) = \ell - 1$; it is the algebra of the \emph{curvilinear} fat point of length $\ell$, whose versal deformation is the smooth $(\ell-1)$-parameter unfolding $t^{\ell} + a_{\ell-2}t^{\ell-2} + \cdots + a_0$ of $t^{\ell}$. The following problem, posed by Christophersen in seminar talks in 2016~\cite{Chr16} and recently taken up in~\cite{Sta25} and~\cite[Remark~4.27]{DLMR24}, predicts that this is the extremal case: if $A$ is a finite-dimensional local $\KK$-algebra of dimension $\ell$, then $\dim \Aut^0_\KK(A) \geq \ell - 1$, with equality if and only if $A \cong \KK[t]/(t^{\ell})$. Through the dictionary above, the problem predicts that smoothing a fat point of length $\ell$ requires at least $\ell-1$ parameters, and that the curvilinear point is the unique minimizer. Christophersen's original motivation was that this bound, together with its equality case, would imply the non-existence of rigid reduced curve singularities~\cite{Chr16}, a long-standing open problem; see~\cite[Conjecture~7.3.3]{Gre20}. This problem remains open in general. The main result of this paper settles Christophersen's problem in the class of monomial algebras, that is, quotients $A = \KK[\XX]/I$ where $I \subset \KK[\XX] = \KK[x_1,\ldots,x_n]$ is a monomial ideal. Throughout, we assume that $I$ has no linear generators, i.e., $x_i \notin I$ for every $1 \leq i \leq n$. This is not a restriction: if $x_i \in I$, then $A$ is isomorphic to a monomial algebra in the remaining $n-1$ variables, so we may always remove such variables.

\begin{theorem*}
Let $A \cong \KK[\XX]/I$, where $I$ is a monomial ideal such that $A$ is finite-dimensional of dimension $\ell$. Then $\dim \Aut^0_\KK(A) \geq \ell - 1$. Moreover, equality holds if and only if $A \cong \KK[t]/(t^{\ell})$.
\end{theorem*}

In fact, we prove a sharper statement: if $I$ is reducible then $\dim \Aut^0_\KK(A) \geq \ell - 1 + n$, and if $I$ is irreducible then $\dim \Aut^0_\KK(A) \geq \ell - 1 + (n-1)$; see \cref{proposition: cota} and~\cref{lemma: cota}, respectively. The bound $\ell - 1 + n$ is sharp, as shown by $I = \mathfrak{m}^2 \subset \KK[x_1,x_2]$, where $\Aut_\KK(A) \cong \GL_2(\KK)$. Moreover, it is attained only for $n = 2$ and $I = (x_1^a, x_1x_2, x_2^b)$ with $a, b \geq 2$, while for $n \geq 3$ it improves to $\dim \Aut^0_\KK(A) \geq \ell + n$, which is sharp for $n = 3$; see \cref{theorem: equality}. Similarly, the bound $\ell - 1 + (n-1)$ in the irreducible case is attained only for $n = 1$ and for $I = (x_1^2, x_2^2)$; see \cref{remark: exact irreducible}.

Since monomial fat points are classically known to be smoothable~\cite{Har66} (see also~\cite[Proposition~4.15]{CEVV09}), the theorem, combined with the dictionary above, admits the following geometric reformulation over $\CC$: the smoothing components of the versal base space of a monomial fat point of length $\ell$ all have dimension at least $\ell - 1$, with equality if and only if the fat point is curvilinear.

Our proof is combinatorial. The starting point is the explicit description of $\Aut_\KK(A)$ for finite-dimensional monomial algebras obtained in~\cite{DLMR24}, where the homogeneous derivations of $A$ with respect to the natural $\ZZ^n$-grading are parametrized by Demazure roots, in the spirit of the toric case of Demazure~\cite{Dem70}. From this description we derive an exact formula expressing $\dim \Aut^0_\KK(A)$ as $\ell - 1$ plus two nonnegative combinatorial terms (\cref{lemma: identity}), from which the theorem follows quickly. The sharper bounds require more work: for them we combine the above description with the notion of the irredundant irreducible decomposition of a monomial ideal; see~\cite[Theorem~1.3.1]{HH11}. 

The present paper is part of a broader program describing finite-dimensional monomial algebras through their automorphism groups. The structure of $\Aut_\KK(A)$ for such algebras was determined in~\cite{DLMR24}, and in~\cite{DLA24} it was shown that a finite-dimensional monomial algebra is recovered from its automorphism group. The present paper settles, within this class, the natural extremality question raised by Christophersen's problem.

The paper is organized as follows. In Section~1 we fix notation and recall the necessary background on monomial ideals, irredundant irreducible decompositions, and the description of homogeneous derivations of monomial algebras in terms of Demazure roots. In Section~2 we read the Demazure roots from the irredundant irreducible decomposition. In Section~3 we prove the exact formula and deduce the theorem, we establish the sharper bounds in the irreducible and reducible cases, and we characterize the cases in which these bounds are attained.

\section{Preliminaries}

In this section we fix notation and recall the basic facts about monomial
algebras, their irreducible decompositions, and their derivation algebras
that will be used throughout the paper. The main references are
\cite{HH11} for the combinatorics of monomial ideals and \cite{DLMR24}
for the description of automorphism groups.

\subsection{Monomial algebras and irreducible decompositions}

Let $\KK$ be an algebraically closed field of characteristic zero and let $\KK[\XX] = \KK[x_1,\ldots,x_n]$. For a vector $\mathbf{a} = (a_1,\ldots,a_n) \in \ZZ^n_{\geq 0}$, we write $\XX^{\mathbf{a}} = x_1^{a_1}\cdots x_n^{a_n}$ for the corresponding monomial. A \emph{monomial ideal} is an ideal $I \subset \KK[\XX]$ generated by monomials, and a \emph{monomial algebra} is a quotient $A = \KK[\XX]/I$ by such an ideal. We assume throughout that $I$ has no linear generators; this is not a restriction, since any monomial algebra with linear generators is isomorphic to one without them after eliminating the corresponding variables. We denote by $G(I) \subset \ZZ^n_{\geq 0}$ the set of exponent vectors of the unique minimal set of monomial generators of $I$; see~\cite[Proposition~1.1.6]{HH11}. We also denote by $\mathfrak{m} = (x_1,\ldots,x_n)$ the maximal ideal of $\KK[\XX]$ and by $\supp(I) = \{\mathbf{a} \in \ZZ^n_{\geq 0} \mid \XX^{\mathbf{a}} \in I\}$ the support of $I$. In this paper, all monomial algebras are assumed to be finite-dimensional as $\KK$-vector spaces, which is equivalent to requiring that $x_i^{a_i} \in G(I)$ for some $a_i \in \ZZ_{\geq 2}$ and each $1 \leq i \leq n$.

For a vector $\mathbf{b}=(b_1,\ldots,b_n) \in \ZZ^n_{\geq 0}$, we write
$$
\mathfrak{m}^{\mathbf{b}} = (x_1^{b_1},\, x_2^{b_2},\, \ldots,\, x_n^{b_n} \mid b_i \geq 1),
$$
i.e., we only include the generators corresponding to indices $i$ with $b_i \geq 1$. A monomial ideal of this form is called \emph{irreducible}. A monomial ideal that is not irreducible is called \emph{reducible}. Note that for $\mathbf{a}, \mathbf{b} \in \ZZ^n_{\geq 1}$, we have $\mathfrak{m}^{\mathbf{a}} \subseteq \mathfrak{m}^{\mathbf{b}}$ if and only if $a_i \geq b_i$ for every $1 \leq i \leq n$.

\begin{theorem}[{\cite[Theorem~1.3.1, Corollary~1.3.2]{HH11}}]
\label{thm:irred_monomial}
Every monomial ideal $I \subset \KK[\XX]$ admits an irredundant
decomposition into irreducible monomial ideals
$$
I = \mathfrak{m}^{\mathbf{b}_1} \cap \cdots \cap \mathfrak{m}^{\mathbf{b}_r},
$$
and such a decomposition is unique up to the order of the components.
\end{theorem}

Recall that the decomposition is called \emph{irredundant} if no component can be omitted without changing the intersection, or equivalently, if $\mathfrak{m}^{\mathbf{b}_i} \not\subseteq \mathfrak{m}^{\mathbf{b}_j}$ for every $i \neq j$. Since $A$ is finite-dimensional and $I$ has no linear generators, we have $\mathbf{b}_j \in \ZZ^n_{\geq 1}$ for every $1 \leq j \leq r$.

For an ideal $I \subset \KK[\XX]$ and an element $f \in \KK[\XX]$, the \emph{colon ideal} $I:(f)$ is defined as $\{g \in \KK[\XX] \mid fg \in I\}$. By~\cite[Proposition~1.2.2]{HH11}, the colon ideal of an irreducible monomial ideal with respect to a variable is given by
\begin{equation}\label{eq:colon_irreducible}
\mathfrak{m}^{\mathbf{b}} : (x_i) = \begin{cases}
\mathfrak{m}^{\mathbf{b}} & \text{if } b_i = 0, \\
\KK[\XX] & \text{if } b_i = 1, \\
\mathfrak{m}^{\mathbf{b} - \mathrm{e}_i} & \text{if } b_i \geq 2.
\end{cases}
\end{equation}

We introduce the following notation. We denote by $\mathcal{M}(I) = \{\mathbf{b}_1,\ldots,\mathbf{b}_r\}$ the set of exponent vectors in the irredundant irreducible decomposition of $I$. For each $\mathbf{b} \in \mathcal{M}(I)$, we write $\widehat{\mathbf{b}} = \mathbf{b} - \sum_{i=1}^n \mathrm{e}_i$ and we set $\widehat{\mathcal{M}}(I) = \{\widehat{\mathbf{b}} \mid \mathbf{b} \in \mathcal{M}(I)\}$. For $\mathbf{b} \in \mathcal{M}(I)$, we denote by $b_j$ its $j$-th component. When $\mathcal{M}(I) = \{\mathbf{b}_1,\ldots,\mathbf{b}_r\}$, we write $b_{ij}$ for the $j$-th component of $\mathbf{b}_i$, and similarly $\widehat{b}_j$ and $\widehat{b}_{ij}$ for the components of $\widehat{\mathbf{b}}$ and $\widehat{\mathbf{b}}_i$ respectively.

\subsection{Derivations and automorphism groups of monomial algebras}

We denote by $\Der(A)$ the Lie algebra of derivations of $A = \KK[\XX]/I$, which coincides with the Lie algebra of the automorphism group $\Aut_\KK(A)$; see~\cite[II,~\textsection 4,~2.3]{DG70}. 

The monomial algebra $A$ inherits a natural $\ZZ^n$-grading from
$\KK[\XX]$, which corresponds to a faithful action of the algebraic torus
$T = \GM^n$. This action induces a weight decomposition of the Lie algebra
$$
\Der(A) = \bigoplus_{\alpha \in \ZZ^n} \mathfrak{g}_\alpha,
$$
where $\mathfrak{g}_\alpha$ denotes the space of homogeneous derivations
of degree~$\alpha$. The weight space $\mathfrak{g}_0$ corresponds to the Lie algebra of the diagonal torus $T = \GM^n$ acting on $A$. Since this action is faithful, $\dim \mathfrak{g}_0 = n$. The degree $\alpha$, and the corresponding weight space $\mathfrak{g}_\alpha$, is called \emph{inner} if $\alpha \in \ZZ^n_{\geq 0}$ and \emph{outer} otherwise.

Let $M = \ZZ^n$ and let $N = \homo(M,\ZZ)$ be its dual lattice. We also set $N_\KK = N \otimes_\ZZ \KK$. We denote by $\{\mathrm{e}_1,\ldots,\mathrm{e}_n\}$ the canonical basis of $M$ and
by $\{\mathrm{e}_1^*,\ldots,\mathrm{e}_n^*\}$ the dual basis of $N$. For $\alpha \in M$ and $p \in N_\KK$, the map
$$
\partial_{\alpha,p} \colon \XX^{\mathbf{a}} \mapsto p(\mathbf{a})\,\XX^{\mathbf{a}+\alpha}
$$
defines a homogeneous derivation of degree $\alpha$ on the Laurent polynomial ring $\KK[\XX^{\pm 1}]$, and every homogeneous derivation of $\KK[\XX^{\pm 1}]$ arises this way by the Leibniz rule; see~\cite[Section~1]{DLMR24}. This derivation restricts to a derivation of $\KK[\XX]$ in two cases; see~\cite[Proposition~3.1]{KLL15} and~\cite[Proposition~1.3]{DLMR24}. First, if $\alpha \in \ZZ^n_{\geq 0}\subset M$, the restriction is well-defined for any $p \in N_\KK$; we call $\partial_{\alpha,p}$ \emph{inner} in this case. Second, if
$$
\alpha \in \mathcal{R}_{\mathrm{e}_i^*} := \left\{\alpha \in \ZZ^n \mid \mathrm{e}_i^*(\alpha) = -1 \text{ and } \mathrm{e}_j^*(\alpha) \geq 0 \text{ for all } j \neq i\right\}
$$
for some $1 \leq i \leq n$, and $p = \lambda \mathrm{e}_i^*$ for some $\lambda \in \KK$, the restriction is also well-defined; we call $\partial_{\alpha,p}$ \emph{outer}. 

Since every derivation of $A$ lifts to a derivation of the polynomial ring $\KK[\XX]$ preserving $I$ (see~\cite[Lemma~2.10 and Remark~2.3]{DLMR24}), every homogeneous derivation of $A$ is induced by a derivation of the form $\partial_{\alpha,p}$. We denote by $\overline{\partial}_{\alpha,p}$ the induced derivation on the quotient $A = \KK[\XX]/I$. Note that $\overline{\partial}_{\alpha,p}$ may be the zero derivation even when $\partial_{\alpha,p}$ is not; this happens if and only if $(\ZZ^n_{\geq 0} \setminus p^\perp) + \alpha \subset \supp(I)$; see~\cite[Lemma~3.9]{DLMR24}. In particular, $\mathfrak{g}_\alpha = 0$ for every $\alpha \in \supp(I)$, since in this case $\mathbf{a} + \alpha \in \supp(I)$ for every $\mathbf{a} \in \ZZ^n_{\geq 0}$.

When $\partial_{\alpha,p}$ is inner, it preserves any monomial ideal. Since $A$ is finite-dimensional, any non-zero $\overline{\partial}_{\alpha,p}$ with $\alpha \in \ZZ^n_{\geq 0} \setminus \{0\}$ is nilpotent; see~\cite[Lemma~3.4]{DLMR24}. When $\partial_{\alpha,p}$ is outer with $\alpha \in \mathcal{R}_{\mathrm{e}_i^*}$, it preserves $I$ if and only if
\begin{equation}\label{eq: outer criterion}
\mathbf{a} + \alpha \in \supp(I) \quad \text{for all } \mathbf{a} \in G(I) \setminus (\mathrm{e}_i^*)^\perp;
\end{equation}
see~\cite[Proposition~1.4]{DLMR24}. In this case, the induced derivation is also nilpotent on $A$; see~\cite[Lemma~3.3]{DLMR24}.

We call an element $\alpha \in M \setminus \{0\}$ a \emph{Demazure root} of $I$ if there exists $p \in N_\KK$ such that $\partial_{\alpha,p}$ induces a nonzero derivation $\overline{\partial}_{\alpha,p}$ of $A$, and we denote by $\mathcal{R}(I)$ the set of all Demazure roots of $I$. We call a Demazure root \emph{inner} if $\alpha \in \ZZ^n_{\geq 0}$ and \emph{outer} otherwise. For each $1 \leq j \leq n$, we set $\mathcal{R}_{\mathrm{e}_j^*}(I) = \mathcal{R}_{\mathrm{e}_j^*} \cap \mathcal{R}(I)$, so that every outer Demazure root belongs to a unique $\mathcal{R}_{\mathrm{e}_j^*}(I)$. For each $\alpha \in \ZZ^n_{\geq 0}$, we define
\begin{equation}\label{eqn def E_alpha}
E_\alpha := \{\mathrm{e}_i \mid \alpha + \mathrm{e}_i \in \ZZ^n_{\geq 0} \setminus \supp(I)\}.
\end{equation}
For each $\alpha \in \mathcal{R}(I)$, we define the $\KK$-vector space
$$
N_\KK(\alpha) = \begin{cases}
\displaystyle\bigoplus_{\mathrm{e}_i \in E_\alpha} \KK \cdot \mathrm{e}_i^* & \text{if } \alpha \text{ is inner}, \\[0.8em]
\KK \cdot \mathrm{e}_j^* & \text{if } \alpha \in \mathcal{R}_{\mathrm{e}_j^*}(I).
\end{cases}
$$

The following lemma shows that $N_\KK(\alpha)$ parametrizes the homogeneous derivations of degree $\alpha$ on $A$.

\begin{lemma}\label{lemma:G_alpha_iso}
Let $I \subset \KK[\XX]$ be a monomial ideal such that $A = \KK[\XX]/I$ is
finite-dimensional. For each $\alpha \in \mathcal{R}(I)$, the linear map
$$
N_\KK(\alpha) \to \mathfrak{g}_\alpha, \quad p \mapsto
\overline{\partial}_{\alpha,p}
$$
is an isomorphism of $\KK$-vector spaces. In particular,
$\dim \mathfrak{g}_\alpha = |E_\alpha|$ for inner roots and
$\dim \mathfrak{g}_\alpha = 1$ for outer roots.
\end{lemma}

\begin{proof}
We first show injectivity. Let $p \in N_\KK(\alpha)$ be nonzero.

If $\alpha$ is outer with $\alpha \in \mathcal{R}_{\mathrm{e}_j^*}(I)$, then $p = \lambda \mathrm{e}_j^*$ for some $\lambda \neq 0$. Since $\alpha$ is a Demazure root, there exists $\mathbf{a} \in \ZZ^n_{\geq 0} \setminus \supp(I)$ with $a_j \geq 1$ such that $\mathbf{a} + \alpha \notin \supp(I)$. As $\supp(I)$ is closed under the partial order and $\mathrm{e}_j + \alpha \leq \mathbf{a} + \alpha$, we get $\mathrm{e}_j + \alpha \notin \supp(I)$. Hence $\overline{\partial}_{\alpha,p}(\overline{x}_j) = \lambda\,\overline{\XX}^{\mathrm{e}_j + \alpha} \neq 0$.

If $\alpha$ is inner, then $p \in N_\KK(\alpha) = \bigoplus_{\mathrm{e}_i \in E_\alpha} \KK \cdot \mathrm{e}_i^*$, so since $p \neq 0$, there exists $\mathrm{e}_i \in E_\alpha$ such that $p(\mathrm{e}_i) \neq 0$. By definition of $E_\alpha$, we have $\alpha + \mathrm{e}_i \notin \supp(I)$, and so $\overline{\partial}_{\alpha,p}(\overline{x}_i) = p(\mathrm{e}_i)\,\overline{\XX}^{\mathrm{e}_i+\alpha} \neq 0$.

In both cases, $\overline{\partial}_{\alpha,p} \neq 0$, proving injectivity.

For surjectivity, let $p \in N_\KK$ be such that $\overline{\partial}_{\alpha,p} \neq 0$. We decompose $p = p_1 + p_2$, where $p_1 \in N_\KK(\alpha)$ and $p_2 \in \bigoplus_{\mathrm{e}_i \notin E_\alpha} \KK \cdot \mathrm{e}_i^*$. If $\alpha$ is outer with $\alpha \in \mathcal{R}_{\mathrm{e}_j^*}(I)$, then $p = \lambda \mathrm{e}_j^*$ for some $\lambda \in \KK$, so $p_2 = 0$. If $\alpha$ is inner, then for each $1 \leq i \leq n$, if $\mathrm{e}_i \in E_\alpha$ then $p_2(\mathrm{e}_i) = 0$ by construction of $N_\KK(\alpha)$, and if $\mathrm{e}_i \notin E_\alpha$ then $\alpha + \mathrm{e}_i \in \supp(I)$, so $\overline{\XX}^{\mathrm{e}_i + \alpha} = 0$ in $A$. In both cases, $\overline{\partial}_{\alpha,p_2}(\overline{x}_i) = 0$, and since derivations are determined by their action on the generators, we have $\overline{\partial}_{\alpha,p_2} = 0$. Hence $\overline{\partial}_{\alpha,p} = \overline{\partial}_{\alpha,p_1}$ with $p_1 \in N_\KK(\alpha)$, proving surjectivity.
\end{proof}

We illustrate the above notions with an example that we will revisit throughout the paper.

\begin{example}\label{example: running}
Let $n = 2$ and
$I = (x_1^3,\, x_1^2 x_2,\, x_1 x_2^3,\, x_2^5) \subset \KK[x_1,x_2]$.
The irredundant irreducible decomposition of $I$ is
$$
I = \mathfrak{m}^{(3,1)} \cap \mathfrak{m}^{(2,3)} \cap
\mathfrak{m}^{(1,5)},
$$
so $\mathcal{M}(I) = \{(3,1),\,(2,3),\,(1,5)\}$ and $A = \KK[x_1,x_2]/I$
has dimension $\ell = |\ZZ^2_{\geq 0} \setminus \supp(I)| = 9$. The
shifted vectors are $\widehat{\mathbf{b}}_1 = (2,0)$,
$\widehat{\mathbf{b}}_2 = (1,2)$, and $\widehat{\mathbf{b}}_3 = (0,4)$.
Note that $\widehat{\mathbf{b}}_1$ lies on the $x_1$-axis and
$\widehat{\mathbf{b}}_3$ lies on the $x_2$-axis. 

The inner Demazure roots and the dimensions of the corresponding spaces $\mathfrak{g}_\alpha$ are:
$$
\begin{array}{c|ccccc}
\alpha & (1,0) & (0,1) & (1,1) & (0,2) & (0,3) \\
\hline
\dim \mathfrak{g}_\alpha & 2 & 2 & 1 & 2 & 1
\end{array}
$$
The outer Demazure roots are: $(1,-1)$ and $(2,-1)$ in $\mathcal{R}_{\mathrm{e}_2^*}(I)$, and $(-1,2)$, $(-1,3)$, $(-1,4)$ in $\mathcal{R}_{\mathrm{e}_1^*}(I)$, each contributing $\dim \mathfrak{g}_\alpha = 1$. In \cref{fig:running-roots}, solid black dots represent $\supp(I)$, squares mark the shifted vectors $\widehat{\mathbf{b}}_j$, colored circles indicate Demazure roots (green for inner and red for outer), the hollow black circle marks the origin, and gray dots mark the remaining lattice points outside $\ZZ^2_{\geq 0}$.

\begin{figure}[H]
\centering
\begin{picture}(100,80)
\definecolor{gray1}{gray}{0.7}
\definecolor{gray2}{gray}{0.85}
\definecolor{green}{RGB}{0,124,0}
\textcolor{gray2}{\put(0,20){\vector(1,0){95}}}
\textcolor{gray2}{\put(10,10){\vector(0,1){68}}}
% row y=-1
\put(0,10){\textcolor{gray1}{\circle*{3}}}
\put(10,10){\textcolor{gray1}{\circle*{3}}}
\put(20,10){\textcolor{red}{\circle{3}}}
\put(30,10){\textcolor{red}{\circle{3}}}
\put(40,10){\textcolor{gray1}{\circle*{3}}}
\put(50,10){\textcolor{gray1}{\circle*{3}}}
\put(60,10){\textcolor{gray1}{\circle*{3}}}
\put(70,10){\textcolor{gray1}{\circle*{3}}}
\put(80,10){\textcolor{gray1}{\circle*{3}}}
% row y=0
\put(0,20){\textcolor{gray1}{\circle*{3}}}
\put(10,20){\circle{3}}
\put(20,20){\textcolor{green}{\circle{3}}}
\put(30,20){\framebox(3,3){}}
\put(40,20){\circle*{3}}
\put(50,20){\circle*{3}}
\put(60,20){\circle*{3}}
\put(70,20){\circle*{3}}
\put(80,20){\circle*{3}}
% row y=1
\put(0,30){\textcolor{gray1}{\circle*{3}}}
\put(10,30){\textcolor{green}{\circle{3}}}
\put(20,30){\textcolor{green}{\circle{3}}}
\put(30,30){\circle*{3}}
\put(40,30){\circle*{3}}
\put(50,30){\circle*{3}}
\put(60,30){\circle*{3}}
\put(70,30){\circle*{3}}
\put(80,30){\circle*{3}}
% row y=2
\put(0,40){\textcolor{red}{\circle{3}}}
\put(10,40){\textcolor{green}{\circle{3}}}
\put(20,40){\framebox(3,3){}}
\put(30,40){\circle*{3}}
\put(40,40){\circle*{3}}
\put(50,40){\circle*{3}}
\put(60,40){\circle*{3}}
\put(70,40){\circle*{3}}
\put(80,40){\circle*{3}}
% row y=3
\put(0,50){\textcolor{red}{\circle{3}}}
\put(10,50){\textcolor{green}{\circle{3}}}
\put(20,50){\circle*{3}}
\put(30,50){\circle*{3}}
\put(40,50){\circle*{3}}
\put(50,50){\circle*{3}}
\put(60,50){\circle*{3}}
\put(70,50){\circle*{3}}
\put(80,50){\circle*{3}}
% row y=4
\put(0,60){\textcolor{red}{\circle{3}}}
\put(10,60){\framebox(3,3){}}
\put(20,60){\circle*{3}}
\put(30,60){\circle*{3}}
\put(40,60){\circle*{3}}
\put(50,60){\circle*{3}}
\put(60,60){\circle*{3}}
\put(70,60){\circle*{3}}
\put(80,60){\circle*{3}}
% row y=5
\put(0,70){\textcolor{gray1}{\circle*{3}}}
\put(10,70){\circle*{3}}
\put(20,70){\circle*{3}}
\put(30,70){\circle*{3}}
\put(40,70){\circle*{3}}
\put(50,70){\circle*{3}}
\put(60,70){\circle*{3}}
\put(70,70){\circle*{3}}
\put(80,70){\circle*{3}}
\end{picture}
\caption{$I = (x_1^3,\, x_1^2 x_2,\, x_1 x_2^3,\, x_2^5)$}
\label{fig:running-roots}
\end{figure}
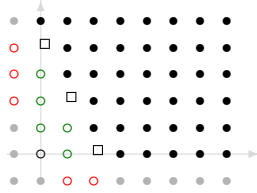

In particular, $\dim \mathfrak{g}_0 = 2$ and $\sum_{\alpha \in \mathcal{R}(I)} \dim \mathfrak{g}_\alpha = 13$, so $\dim \Aut^0_\KK(A) = 2+13 = 15 > 8 = \ell - 1$.
\end{example}

\section{Demazure roots via irreducible decompositions}

In this section we read the Demazure roots of $I$ from its irredundant irreducible decomposition. We characterize the vanishing of the sets $E_\alpha$ defined in~\eqref{eqn def E_alpha} in terms of the shifted vectors $\widehat{\mathcal{M}}(I)$, and we construct outer Demazure roots from the geometry of $\widehat{\mathcal{M}}(I)$. Both results will be needed to establish the sharper bound in the reducible case.

We start with a standard property of the colon ideal that will be used in the next proposition.

\begin{lemma}\label{lemma: intersection}
Let $I, J \subset \KK[\XX]$ be monomial ideals. Then
$(I \cap J) : (x_i) = (I : (x_i)) \cap (J : (x_i))$.
\end{lemma}

\begin{proof}
An element $f$ belongs to $(I \cap J) : (x_i)$ if and only if
$x_i f \in I \cap J$, which holds if and only if $x_i f \in I$ and
$x_i f \in J$, i.e., $f \in I : (x_i)$ and $f \in J : (x_i)$.
\end{proof}

The next proposition characterizes the degrees $\alpha \in \ZZ^n_{\geq 0} \setminus \supp(I)$ for which $E_\alpha = \emptyset$, equivalently $\mathfrak{g}_\alpha = 0$, in terms of the shifted vectors $\widehat{\mathcal{M}}(I)$.

\begin{proposition}\label{proposition: empty}
Let $I$ be a monomial ideal with irredundant irreducible decomposition
$I = \mathfrak{m}^{\mathbf{b}_1} \cap \cdots \cap \mathfrak{m}^{\mathbf{b}_r}$. For
$\alpha \in \ZZ^n_{\geq 0} \setminus \supp(I)$, we have
$E_\alpha = \emptyset$ if and only if
$\alpha = \widehat{\mathbf{b}}_j$ for some $1 \leq j \leq r$.
\end{proposition}

\begin{proof}
Let $\alpha \in \ZZ^n_{\geq 0} \setminus \supp(I)$. By definition,
$\mathrm{e}_i \in E_\alpha$ if and only if $\alpha + \mathrm{e}_i \notin \supp(I)$, which is equivalent to $\XX^{\alpha + \mathrm{e}_i} \notin I$. By definition of the colon ideal, this holds if and only if $\XX^\alpha \notin I : (x_i)$. Therefore, $E_\alpha = \emptyset$ if and only if $\XX^\alpha \in I : (x_i)$ for every $1 \leq i \leq n$, i.e.,
$$
\XX^\alpha \in \bigcap_{i=1}^n (I : (x_i)).
$$

By \cref{lemma: intersection} and~\eqref{eq:colon_irreducible}, we have
$$
\bigcap_{i=1}^n (I : (x_i)) = \bigcap_{j=1}^r \bigcap_{i=1}^n (\mathfrak{m}^{\mathbf{b}_j} : (x_i)).
$$
We denote this ideal by $\widehat{I}$. We seek $\alpha \in \ZZ^n_{\geq 0}$ such that $\XX^\alpha \notin I$ and $\XX^\alpha \in \widehat{I}$.

Since $\mathbf{b}_j \in \ZZ^n_{\geq 1}$ for every $j$, the condition $\XX^\alpha \notin I$ means that there exists $j$ such that $\XX^\alpha \notin \mathfrak{m}^{\mathbf{b}_j}$, i.e., $\alpha_i < b_{ji}$ for every $1 \leq i \leq n$. On the other hand, $\XX^\alpha \in \widehat{I}$ implies in particular that $\XX^\alpha \in \bigcap_{i=1}^n (\mathfrak{m}^{\mathbf{b}_j} : (x_i))$. We claim that $\alpha_i \geq b_{ji} - 1$ for every $i$. Indeed, if $b_{ji} = 1$, then $\alpha_i < b_{ji} = 1$ already gives $\alpha_i = 0 = b_{ji} - 1$. If $b_{ji} \geq 2$, then~\eqref{eq:colon_irreducible} gives $\XX^\alpha \in \mathfrak{m}^{\mathbf{b}_j - \mathrm{e}_i}$, so $\XX^\alpha$ is divisible by some generator $x_k^{(b_j-\mathrm{e}_i)_k}$ with $(b_j-\mathrm{e}_i)_k \geq 1$. If $k \neq i$, this would give $\alpha_k \geq b_{jk}$, contradicting $\alpha_k < b_{jk}$. Hence $\alpha_i \geq b_{ji} - 1$. Combined with $\alpha_i < b_{ji}$, we obtain $\alpha_i = b_{ji} - 1$ for every $i$, i.e.,
$\alpha = \widehat{\mathbf{b}}_j$.

Conversely, let $\alpha = \widehat{\mathbf{b}}_j$ for some $j$. Since $(\widehat{\mathbf{b}}_j)_i = b_{ji} - 1 < b_{ji}$ for every $i$, we have $\XX^{\widehat{\mathbf{b}}_j} \notin \mathfrak{m}^{\mathbf{b}_j}$, and hence $\XX^{\widehat{\mathbf{b}}_j} \notin I$. For each $q \neq j$, if $\XX^{\widehat{\mathbf{b}}_j} \notin \mathfrak{m}^{\mathbf{b}_q}$, then $b_{ji} \leq b_{qi}$ for every $i$, giving $\mathfrak{m}^{\mathbf{b}_q} \subseteq \mathfrak{m}^{\mathbf{b}_j}$, contradicting irredundancy. Hence $\XX^{\widehat{\mathbf{b}}_j} \in \mathfrak{m}^{\mathbf{b}_q}$ for all $q \neq j$. Now, for each $i$, the $i$-th component of $\widehat{\mathbf{b}}_j + \mathrm{e}_i$ is $b_{ji} \geq 1$, so $\XX^{\widehat{\mathbf{b}}_j + \mathrm{e}_i} \in \mathfrak{m}^{\mathbf{b}_j}$. Combined with the above, this gives $\widehat{\mathbf{b}}_j + \mathrm{e}_i \in \supp(I)$ for every $i$, i.e., $E_{\widehat{\mathbf{b}}_j} = \emptyset$.
\end{proof}

The next lemma constructs outer Demazure roots from the shifted vectors $\widehat{\mathbf{b}} \in \widehat{\mathcal{M}}(I)$ that lie on a proper face of $\ZZ^n_{\geq 0}$.

\begin{lemma}\label{lemma: outer}
Let $I$ be a reducible monomial ideal, let $\widehat{\mathbf{b}} \in \widehat{\mathcal{M}}(I)$, and let $F \neq \{0\}$ be the minimal face of $\ZZ^n_{\geq 0}\subset M$ containing $\widehat{\mathbf{b}}$. Let $K \subset \{1,\ldots,n\}$ be the subset of indices $i$ such that $\mathrm{e}_i \in F$, and set $K' = \{1,\ldots,n\} \setminus K$. Then, for each $i \in K'$, the element
$$
\alpha = \sum_{j \in K} \widehat{b}_j \mathrm{e}_j - \mathrm{e}_i
$$
belongs to $\mathcal{R}(I)$.
\end{lemma}

\begin{proof}
Without loss of generality, we may assume that there exists $1 \leq s < n$ such that $K = \{1,\ldots,s\}$. Then $\widehat{b}_j > 0$ for every $j \in K$. For $i \in K'$, consider the derivation
$$
\partial = x_1^{\widehat{b}_1} \cdots x_s^{\widehat{b}_s} \frac{d}{dx_i}.
$$
We need to verify that $\partial(I) \subset I$ and that the induced derivation $\overline{\partial}$ on $A$ is nonzero.

The derivation $\overline{\partial}$ is nonzero because $\overline{\partial}(\overline{x}_i) = \overline{x}_1^{\widehat{b}_1} \cdots \overline{x}_s^{\widehat{b}_s} \neq 0$, since $\widehat{\mathbf{b}} \notin \supp(I)$, as $\XX^{\widehat{\mathbf{b}}} \notin \mathfrak{m}^{\mathbf{b}} \supseteq I$.

It remains to show that $\partial(I) \subset I$. Let $\mathbf{c} \in G(I)$. If $c_i = 0$, then $\partial(\XX^{\mathbf{c}}) = 0$. If $c_i \geq 1$, then since $I$ has no linear generators, $\mathbf{c} \neq \mathrm{e}_i$ and $\partial(\XX^{\mathbf{c}}) = \lambda\, x_1^{\widehat{b}_1} \cdots x_s^{\widehat{b}_s}\, \XX^{\mathbf{c} - \mathrm{e}_i}$ for some $\lambda \in \KK^*$. We claim that this monomial belongs to $I$.

For each $\mathbf{b}_q \in \mathcal{M}(I)$ with $\mathbf{b}_q \neq \mathbf{b}$, there exists $j_q \in K$ such that $b_{q,j_q} \leq \widehat{b}_{j_q} < b_{j_q}$; otherwise $b_{q,j} \geq b_j$ for every $j \in K$ and $b_{q,j} \geq 1 = b_j$ for every $j \in K'$, giving $\mathfrak{m}^{\mathbf{b}_q} \subseteq \mathfrak{m}^{\mathbf{b}}$, contradicting irredundancy. If $\mathbf{c} - \mathrm{e}_i$ has a nonzero component in some coordinate $s' \in K'$, then $x_1^{\widehat{b}_1} \cdots x_s^{\widehat{b}_s}\, \XX^{\mathbf{c} - \mathrm{e}_i}$ lies in $\mathfrak{m}^{\mathbf{b}}$ (since $b_{s'} = 1$) and in each $\mathfrak{m}^{\mathbf{b}_q}$ (since it is divisible by $x_{j_q}^{b_{q,j_q}}$), so it belongs to $I$. Otherwise, $\mathbf{c} - \mathrm{e}_i$ is supported on $K$ and is nonzero (since $I$ has no linear generators), so there exists $j_0 \in K$ with $c_{j_0} \geq 1$. Then the exponent of $x_1^{\widehat{b}_1} \cdots x_s^{\widehat{b}_s}\, \XX^{\mathbf{c} - \mathrm{e}_i}$ dominates $\widehat{\mathbf{b}} + \mathrm{e}_{j_0}$, which belongs to $\supp(I)$ since $E_{\widehat{\mathbf{b}}} = \emptyset$ by \cref{proposition: empty}. As $\supp(I)$ is closed under the partial order, the monomial $x_1^{\widehat{b}_1} \cdots x_s^{\widehat{b}_s}\, \XX^{\mathbf{c} - \mathrm{e}_i} \in I$.
\end{proof}

We illustrate \cref{lemma: outer} by revisiting \cref{example: running}.

\begin{example}[Continuation of \cref{example: running}]
For $\widehat{\mathbf{b}}_1 = (2,0)$, the minimal face of $\ZZ^2_{\geq 0}$ containing it is the $x_1$-axis, so $K = \{1\}$ and $K' = \{2\}$. \cref{lemma: outer} gives the outer root $\alpha = 2\mathrm{e}_1 - \mathrm{e}_2 = (2,-1)$, corresponding to the derivation $x_1^2 \frac{d}{dx_2}$. Similarly, for $\widehat{\mathbf{b}}_3 = (0,4)$, the minimal face is the $x_2$-axis, $K = \{2\}$ and $K' = \{1\}$, yielding the outer root $\alpha = 4\mathrm{e}_2 - \mathrm{e}_1 = (-1,4)$, corresponding to $x_2^4 \frac{d}{dx_1}$. For $\widehat{\mathbf{b}}_2 = (1,2)$, the minimal face is all of $\ZZ^2_{\geq 0}$, so $K = \{1,2\}$ and $K' = \emptyset$; in this case \cref{lemma: outer} produces no outer root.
\end{example}

\begin{remark}\label{remark: 0}
The hypothesis $F \neq \{0\}$ in \cref{lemma: outer} is automatically satisfied under our assumptions. Indeed, $F = \{0\}$ would force $\widehat{\mathbf{b}} = 0$, hence $\mathbf{b} = \sum_{i=1}^n \mathrm{e}_i$ and $\mathfrak{m}^{\mathbf{b}} = \mathfrak{m}$, which is excluded since $I$ has no linear generators.
\end{remark}

\section{Proof of the theorem and sharper bounds}

We start with an exact formula for $\dim \Aut^0_\KK(A)$, from which the theorem follows quickly. We then establish sharper bounds in the irreducible and reducible cases, and we characterize the cases in which these bounds are attained.

\subsection{An exact formula}

For $\mathbf{m} \in \ZZ^n_{\geq 0}$ we write $\nu(\mathbf{m}) = \#\{i \mid m_i \geq 1\}$ for the number of variables dividing $\XX^{\mathbf{m}}$, and we denote by $s(I)$ the number of outer Demazure roots of $I$.

\begin{lemma}\label{lemma: identity}
Let $I \subset \KK[\XX]$ be a monomial ideal such that $A = \KK[\XX]/I$ is finite-dimensional of dimension $\ell$. Then
$$
\dim \Aut^0_\KK(A) = \ell - 1 + s(I) + P(I), \qquad \text{where} \quad P(I) = \sum_{\mathbf{m} \in (\ZZ^n_{\geq 0} \setminus \supp(I)) \setminus \{0\}} \bigl(\nu(\mathbf{m}) - 1\bigr).
$$
\end{lemma}

\begin{proof}
We claim that $\dim \mathfrak{g}_\alpha = |E_\alpha|$ for every $\alpha \in \ZZ^n_{\geq 0}$. For $\alpha = 0$ both sides equal $n$, since $I$ has no linear generators. For $\alpha \in \supp(I)$ both sides vanish. Now let $\alpha \in \ZZ^n_{\geq 0} \setminus \supp(I)$ with $\alpha \neq 0$. If $\alpha$ is a Demazure root, the claim is \cref{lemma:G_alpha_iso}. If $\alpha$ is not a Demazure root, then $\mathfrak{g}_\alpha = 0$, and also $E_\alpha = \emptyset$: otherwise, for $\mathrm{e}_i \in E_\alpha$, the inner derivation $\partial_{\alpha,\mathrm{e}_i^*}$, which preserves $I$, would induce a derivation of $A$ with $\overline{\partial}_{\alpha,\mathrm{e}_i^*}(\overline{x}_i) = \overline{\XX}^{\alpha + \mathrm{e}_i} \neq 0$, and $\alpha$ would be a Demazure root. Hence
$$
\dim \Aut^0_\KK(A) = \sum_{\alpha \in \ZZ^n_{\geq 0} \setminus \supp(I)} |E_\alpha| + s(I).
$$
The map $(\alpha, \mathrm{e}_i) \mapsto (\alpha + \mathrm{e}_i, i)$ is a bijection from $\{(\alpha, \mathrm{e}_i) \mid \alpha \in \ZZ^n_{\geq 0} \setminus \supp(I),\ \mathrm{e}_i \in E_\alpha\}$ onto the set of pairs $(\mathbf{m}, i)$ with $\mathbf{m} \in \ZZ^n_{\geq 0} \setminus \supp(I)$ and $m_i \geq 1$. Indeed, let $(\alpha, \mathrm{e}_i)$ be a pair with $\alpha \in \ZZ^n_{\geq 0} \setminus \supp(I)$ and $\mathrm{e}_i \in E_\alpha$. By definition of $E_\alpha$ we have $\alpha + \mathrm{e}_i \notin \supp(I)$, and the $i$-th component of $\alpha + \mathrm{e}_i$ is at least $1$, so the map is well defined. Conversely, let $(\mathbf{m}, i)$ be a pair with $\mathbf{m} \in \ZZ^n_{\geq 0} \setminus \supp(I)$ and $m_i \geq 1$. Then $\mathbf{m} - \mathrm{e}_i \in \ZZ^n_{\geq 0}$, and $\mathbf{m} - \mathrm{e}_i \notin \supp(I)$ because $\mathbf{m} - \mathrm{e}_i \leq \mathbf{m}$ and $\supp(I)$ is closed under the partial order. Moreover, $\mathrm{e}_i \in E_{\mathbf{m} - \mathrm{e}_i}$ because $(\mathbf{m} - \mathrm{e}_i) + \mathrm{e}_i = \mathbf{m} \notin \supp(I)$. Hence $(\mathbf{m}, i) \mapsto (\mathbf{m} - \mathrm{e}_i, \mathrm{e}_i)$ is the inverse map. Therefore
$$
\sum_{\alpha} |E_\alpha| = \sum_{\mathbf{m} \in \ZZ^n_{\geq 0} \setminus \supp(I)} \nu(\mathbf{m}) = (\ell - 1) + P(I),
$$
since exactly $\ell - 1$ elements of $\ZZ^n_{\geq 0} \setminus \supp(I)$ are different from $0$, and $\nu(0) = 0$.
\end{proof}

We can now prove the main result of this paper, stated in the introduction.

\begin{theorem}\label{theorem: main}
Let $A \cong \KK[\XX]/I$, where $I$ is a monomial ideal such that $A$ is finite-dimensional of dimension $\ell$. Then $\dim \Aut^0_\KK(A) \geq \ell - 1$. Moreover, equality holds if and only if $A \cong \KK[t]/(t^{\ell})$.
\end{theorem}

\begin{proof}
If $\ell = 1$, then $A \cong \KK \cong \KK[t]/(t)$ and $\dim \Aut^0_\KK(A) = 0 = \ell - 1$, so the result holds. We may therefore assume $\ell \geq 2$. By \cref{lemma: identity}, $\dim \Aut^0_\KK(A) = \ell - 1 + s(I) + P(I) \geq \ell - 1$, which proves the inequality.

For the characterization of equality, we first observe that for $\ell \geq 2$,
$$
\Aut^0_\KK\!\bigl(\KK[t]/(t^{\ell})\bigr) \cong
\begin{cases}
\GM, & \ell = 2, \\[4pt]
L_{\ell-2} \rtimes \GM, & \ell \geq 3,
\end{cases}
$$
where $L_{\ell-2}$ is the $(\ell-2)$-dimensional filiform unipotent group; see~\cite[Theorem~4.21 and Example~4.23]{DLMR24}. In both cases, $\dim \Aut^0_\KK(\KK[t]/(t^\ell)) = \ell - 1$.

Conversely, suppose that $\dim \Aut^0_\KK(A) = \ell - 1$. By \cref{lemma: identity}, $s(I) = P(I) = 0$. The condition $P(I) = 0$ means that $\nu(\mathbf{m}) \leq 1$ for every $\mathbf{m} \in \ZZ^n_{\geq 0} \setminus \supp(I)$, i.e., $x_ix_j \in I$ for all $i \neq j$. Assume that $n \geq 2$, and let $x_1^{a_1}, x_2^{a_2} \in G(I)$, with $a_1, a_2 \geq 2$. Then $G(I)$ consists of the pure powers $x_i^{a_i}$ together with the monomials $x_ix_j$ for $i \neq j$. Consider the derivation $\partial = x_1^{a_1 - 1} \frac{\partial}{\partial x_2}$. We have $\partial(x_2^{a_2}) = a_2 x_1^{a_1-1} x_2^{a_2-1} \in (x_1x_2) \subset I$, $\partial(x_1x_2) = x_1^{a_1} \in I$, and $\partial(x_2x_k) = x_1^{a_1-1}x_k \in (x_1x_k) \subset I$ for $k \geq 3$, while $\partial$ vanishes on the remaining generators. Hence $\partial(I) \subset I$. Moreover, the induced derivation maps $\overline{x}_2$ to $\overline{x}_1^{a_1-1} \neq 0$, so $(a_1 - 1)\mathrm{e}_1 - \mathrm{e}_2$ is an outer Demazure root of $I$, contradicting $s(I) = 0$. Therefore $n = 1$ and $A \cong \KK[t]/(t^\ell)$.
\end{proof}

The lower bound in \cref{theorem: main} depends only on $\ell$. A natural question is whether, for monomial algebras, one can obtain sharper lower bounds that reflect the interplay between the dimension $\ell$ of $A$ and the number $n$ of variables, which equals the embedding dimension $\dim_\KK \mathfrak{m}_A/\mathfrak{m}_A^2$ of $A$, where $\mathfrak{m}_A = \mathfrak{m}/I$ denotes the maximal ideal of $A$, since $I$ has no linear generators. We address this question in the rest of this section. We first treat the irreducible case, where we show that $\dim \Aut^0_\KK(A) \geq \ell - 1 + (n-1)$ and compute $\dim \Aut^0_\KK(A)$ explicitly. We then treat the reducible case, where we show that $\dim \Aut^0_\KK(A) \geq \ell - 1 + n$ and determine all the ideals attaining this bound.

\subsection{The irreducible case}

\begin{lemma}\label{lemma: cota}
Let $I = \mathfrak{m}^{\mathbf{b}}$ be an irreducible monomial ideal such that $A = \KK[\XX]/I$ is finite-dimensional of dimension $\ell$. Then
$$
\dim \Aut^0_\KK(A) \geq \ell - 1 + (n-1).
$$
In particular, $\dim \Aut^0_\KK(A) > \ell - 1$ when $n \geq 2$.
\end{lemma}

\begin{proof}
Since $I = \mathfrak{m}^{\mathbf{b}}$ is irreducible with
$b_i \geq 2$ for every $i$, there are no outer Demazure
roots. Indeed, by~\cite[Proposition~1.4]{DLMR24}, an outer root
$\alpha$ with $\mathrm{e}_j^*(\alpha) = -1$ would require
$b_j \mathrm{e}_j + \alpha \in \supp(I)$, forcing
$\alpha_k \geq b_k$ for some $k \neq j$, but then
$\mathrm{e}_j + \alpha \in \supp(I)$ and the induced
derivation would be zero. Hence $s(I) = 0$, and \cref{lemma: identity} gives $\dim \Aut^0_\KK(A) = \ell - 1 + P(I)$. If $n = 1$, this already gives $\dim \Aut^0_\KK(A) \geq \ell - 1 = \ell - 1 + (n - 1)$. If $n \geq 2$, then $\mathrm{e}_i + \mathrm{e}_j \notin \supp(I)$ for every $i < j$, since its components are at most $1 < 2 \leq b_k$. Hence $P(I) \geq \binom{n}{2} \geq n - 1$, and so $\dim \Aut^0_\KK(A) \geq \ell - 1 + (n - 1)$.
\end{proof}

\begin{proposition}\label{remark: exact irreducible}
Let $I = \mathfrak{m}^{\mathbf{b}}$ with $\mathbf{b} \in \ZZ^n_{\geq 2}$, and let $\ell = b_1 \cdots b_n$. Then
$$
\dim \Aut^0_\KK(A) = \ell \sum_{i=1}^n \Bigl(1 - \frac{1}{b_i}\Bigr).
$$
In particular, the bound of \cref{lemma: cota} is attained if and only if $n = 1$ or $\mathbf{b} = (2,2)$.
\end{proposition}

\begin{proof}
Since there are no outer Demazure roots, the proof of \cref{lemma: identity} gives $\dim \Aut^0_\KK(A) = \sum_{\alpha} |E_\alpha|$, where $\alpha$ runs over $\ZZ^n_{\geq 0} \setminus \supp(I)$. This set consists of the $\alpha$ with $\alpha_i \leq b_i - 1$ for every $i$, and for such $\alpha$ we have $\mathrm{e}_i \in E_\alpha$ if and only if $\alpha_i \leq b_i - 2$. Reordering the sum according to the elements of $E_\alpha$, we obtain
$$
\sum_{\alpha} |E_\alpha| = \sum_{i=1}^n \#\bigl\{\alpha \in \ZZ^n_{\geq 0} \setminus \supp(I) \bigm| \mathrm{e}_i \in E_\alpha\bigr\},
$$
since each $\alpha$ is counted on the right-hand side exactly once for each $\mathrm{e}_i \in E_\alpha$. For a fixed $i$, the elements of the set on the right-hand side are exactly the $\alpha$ with $0 \leq \alpha_i \leq b_i - 2$ and $0 \leq \alpha_j \leq b_j - 1$ for every $j \neq i$, so this set has
$$
(b_i - 1)\prod_{j \neq i} b_j = \frac{b_i - 1}{b_i}\, b_1 \cdots b_n = \ell\,\Bigl(1 - \frac{1}{b_i}\Bigr)
$$
elements. Summing over $i$ gives the formula.

For the last statement, note that the bound of \cref{lemma: cota} is $\ell - 1 + (n - 1) = \ell + n - 2$. If $n = 1$, then $\ell = b_1$ and the formula gives $\ell\,(1 - 1/\ell) = \ell - 1$, so the bound is attained. If $n = 2$, the bound is $\ell = b_1b_2$, and the formula gives $\dim \Aut^0_\KK(A) = 2b_1b_2 - b_1 - b_2$. Hence $\dim \Aut^0_\KK(A) - \ell = b_1b_2 - b_1 - b_2 = (b_1 - 1)(b_2 - 1) - 1$, which vanishes if and only if $b_1 - 1 = b_2 - 1 = 1$, since $b_1, b_2 \geq 2$, that is, if and only if $\mathbf{b} = (2,2)$. If $n \geq 3$, then $1 - 1/b_i \geq 1/2$ for every $i$, since $b_i \geq 2$, so the formula gives $\dim \Aut^0_\KK(A) \geq n\ell/2$. Therefore
$$
\dim \Aut^0_\KK(A) - (\ell + n - 2) \geq \frac{n\ell}{2} - \ell - n + 2 = (n - 2)\Bigl(\frac{\ell}{2} - 1\Bigr) > 0,
$$
since $n \geq 3$ and $\ell = b_1 \cdots b_n \geq 2^n > 2$.
\end{proof}

\subsection{The reducible case}

The next lemma produces, from each interior $\widehat{\mathbf{b}} \in \widehat{\mathcal{M}}(I)$, an inner Demazure root with $\dim \mathfrak{g}_\alpha \geq 2$. This extra contribution will be the key to obtaining a strict bound in the reducible case.

\begin{lemma}\label{lemma: interior}
Let $I$ be a reducible monomial ideal and let
$\widehat{\mathbf{b}} \in \widehat{\mathcal{M}}(I)$ be interior,
i.e., with all components positive. Then there exists
$\alpha \in (\ZZ^n_{\geq 0} \setminus \supp(I)) \setminus
(\widehat{\mathcal{M}}(I) \cup \{0\})$ with
$\dim \mathfrak{g}_\alpha \geq 2$.
\end{lemma}

\begin{proof}
We distinguish two cases.

\smallskip
\noindent\textit{Case 1:}
$\widehat{\mathbf{b}} \neq \sum_{i=1}^n \mathrm{e}_i$. Set
$\alpha = \widehat{\mathbf{b}} - \sum_{i=1}^n \mathrm{e}_i$.
Since $\widehat{\mathbf{b}}$ is interior and
$\widehat{\mathbf{b}} \neq \sum_{i=1}^n \mathrm{e}_i$, we have
$\alpha \in \ZZ^n_{\geq 0} \setminus \{0\}$. Since
$\alpha \leq \widehat{\mathbf{b}}$ and
$\widehat{\mathbf{b}} \notin \supp(I)$, the upward closure of
$\supp(I)$ gives $\alpha \notin \supp(I)$. If
$\alpha = \widehat{\mathbf{b}}_k$ for some $k$, then
$\mathbf{b}_k = \mathbf{b} - \sum_{i=1}^n \mathrm{e}_i$, so
$\mathbf{b}_k < \mathbf{b}$ componentwise, giving
$\mathfrak{m}^{\mathbf{b}} \subset \mathfrak{m}^{\mathbf{b}_k}$,
contradicting irredundancy. Hence
$\alpha \notin \widehat{\mathcal{M}}(I)$.

We show $|E_\alpha| = n$. For each $1 \leq i \leq n$, the vector
$\alpha + \mathrm{e}_i = \widehat{\mathbf{b}} -
\sum_{j=1}^n \mathrm{e}_j + \mathrm{e}_i$ has $m$-th component
$\widehat{b}_m - 1 + \delta_{mi} \leq \widehat{b}_m < b_m$
for every $m$. Hence
$\alpha + \mathrm{e}_i \notin
\supp(\mathfrak{m}^{\mathbf{b}}) \supseteq \supp(I)$, so
$\mathrm{e}_i \in E_\alpha$. By \cref{lemma:G_alpha_iso},
$\dim \mathfrak{g}_\alpha = n \geq 2$.

\smallskip
\noindent\textit{Case 2:}
$\widehat{\mathbf{b}} = \sum_{i=1}^n \mathrm{e}_i$, i.e.,
$\mathbf{b} = 2\sum_{i=1}^n \mathrm{e}_i$. By irredundancy,
there exists $k$ such that
$\mathfrak{m}^{\mathbf{b}} \not\subset
\mathfrak{m}^{\mathbf{b}_k}$; otherwise
$\mathfrak{m}^{\mathbf{b}_k}$ would be redundant.
Since $b_m = 2$ for every $m$, this requires
$b_{k,i_0} > 2$ for some index $i_0$. Set
$\alpha = \mathrm{e}_{i_0}$. Since $I$ has no linear generators,
$\alpha \notin \supp(I)$. Since $|E_\alpha| \geq 1$ (shown below),
$\alpha \notin \widehat{\mathcal{M}}(I)$ by
\cref{proposition: empty}.

We show $|E_{\mathrm{e}_{i_0}}| \geq 2$. For each $j \neq i_0$,
the vector $\mathrm{e}_{i_0} + \mathrm{e}_j$ has all components
at most $1 < 2 = b_m$, so
$\mathrm{e}_{i_0} + \mathrm{e}_j \notin
\supp(\mathfrak{m}^{\mathbf{b}}) \supseteq \supp(I)$, giving
$\mathrm{e}_j \in E_{\mathrm{e}_{i_0}}$. Additionally,
$2\mathrm{e}_{i_0}$ has $i_0$-th component
$2 < b_{k,i_0}$ and all other components $0 < b_{km}$, so
$2\mathrm{e}_{i_0} \notin
\supp(\mathfrak{m}^{\mathbf{b}_k}) \supseteq \supp(I)$, giving
$\mathrm{e}_{i_0} \in E_{\mathrm{e}_{i_0}}$. Hence
$|E_{\mathrm{e}_{i_0}}| \geq 2$.
\end{proof}

\begin{remark}\label{remark: case 2 unique}
Note that $\widehat{\mathbf{b}} = \sum_{i=1}^n \mathrm{e}_i$ determines $\mathbf{b}$ uniquely, so Case 2 in the proof of \cref{lemma: interior} occurs for at most one element of $\widehat{\mathcal{M}}(I)$. This observation will be used in the proof of \cref{proposition: cota}.
\end{remark}

We can now establish the bound in the reducible case.

\begin{proposition}\label{proposition: cota}
Let $I$ be a reducible monomial ideal such that $A = \KK[\XX]/I$ is
finite-dimensional of dimension $\ell$. Then
$$
\dim \Aut^0_\KK(A) \geq \ell - 1 + n.
$$
In particular, $\dim \Aut^0_\KK(A) > \ell - 1$.
\end{proposition}
\begin{proof}
Let $I = \mathfrak{m}^{\mathbf{b}_1} \cap \cdots \cap \mathfrak{m}^{\mathbf{b}_r}$ be the irredundant irreducible decomposition of $I$ with $r \geq 2$. We enumerate $\ZZ^n_{\geq 0} \setminus \supp(I) = \{0 = \alpha_1, \alpha_2, \ldots, \alpha_\ell\}$ and denote by $s$ the number of outer Demazure roots of $I$. Then
$$
\dim \Aut^0_\KK(A) = n + \sum_{i=2}^{\ell} \dim \mathfrak{g}_{\alpha_i} + s.
$$
By \cref{proposition: empty}, $\dim \mathfrak{g}_{\alpha_i} = 0$ if and only if $\alpha_i = \widehat{\mathbf{b}}_j$ for some $j$, and the remaining $\ell - 1 - r$ inner roots each satisfy $\dim \mathfrak{g}_{\alpha_i} \geq 1$. Among the $\widehat{\mathbf{b}}_j$, let $r'$ be the number lying on a proper face of $\ZZ^n_{\geq 0}$. By \cref{lemma: outer}, $s \geq r'$, so
$$
\dim \Aut^0_\KK(A) \geq n + r' + (\ell - 1 - r) = \ell - 1 + n - (r - r').
$$
The remaining $r - r'$ elements of $\widehat{\mathcal{M}}(I)$ are interior. By \cref{lemma: interior}, each interior $\widehat{\mathbf{b}}_q$ produces an element $\alpha_q \in (\ZZ^n_{\geq 0} \setminus \supp(I)) \setminus (\widehat{\mathcal{M}}(I) \cup \{0\})$ with $\dim \mathfrak{g}_{\alpha_q} \geq 2$. We claim that the $\alpha_q$ are pairwise different. The map $\widehat{\mathbf{b}}_q \mapsto \widehat{\mathbf{b}}_q - \sum_i \mathrm{e}_i$ is injective, so different interior $\widehat{\mathbf{b}}_q$ in Case 1 of \cref{lemma: interior} produce different $\alpha_q$. By \cref{remark: case 2 unique}, Case 2 occurs for at most one element of $\widehat{\mathcal{M}}(I)$. It remains to check that the Case 2 image $\mathrm{e}_{i_0}$ cannot equal a Case 1 image: if $\mathrm{e}_{i_0} = \widehat{\mathbf{b}}_{q'} - \sum_i \mathrm{e}_i$ for some $\widehat{\mathbf{b}}_{q'}$ in Case 1, then $\mathbf{b}_{q'} = 2 \sum_i \mathrm{e}_i + \mathrm{e}_{i_0} \geq \mathbf{b}_q$ componentwise, giving $\mathfrak{m}^{\mathbf{b}_{q'}} \subseteq \mathfrak{m}^{\mathbf{b}_q}$, contradicting irredundancy.

Let $\Omega$ be the set of elements of $(\ZZ^n_{\geq 0} \setminus \supp(I)) \setminus (\widehat{\mathcal{M}}(I) \cup \{0\})$ different from the $\alpha_q$, so that $|\Omega| = \ell - 1 - r - (r - r')$. Since the elements of $\widehat{\mathcal{M}}(I)$ contribute $0$, we have
$$
\dim \Aut^0_\KK(A) = n + s + \sum_{q} \dim \mathfrak{g}_{\alpha_q} + \sum_{\alpha \in \Omega} \dim \mathfrak{g}_\alpha,
$$
and therefore
\begin{equation}\label{eq: slack decomposition}
\dim \Aut^0_\KK(A) - (\ell - 1 + n) = (s - r') + \sum_{q} \bigl(\dim \mathfrak{g}_{\alpha_q} - 2\bigr) + \sum_{\alpha \in \Omega} \bigl(\dim \mathfrak{g}_\alpha - 1\bigr).
\end{equation}
Each summand on the right-hand side is nonnegative, so $\dim \Aut^0_\KK(A) \geq \ell - 1 + n$.
Since $I$ is reducible, $n \geq 2$, so in particular $\dim \Aut^0_\KK(A) > \ell - 1$.
\end{proof}

\begin{remark}\label{remark: sharp}
The bound $\ell - 1 + n$ is sharp. For $n = 2$ and $I = (\mathfrak{m})^2$, we have $\Aut_\KK(\KK[\XX]/I) \cong \GL_2(\KK)$, which has dimension $4$. In this case $\ell = \dim_\KK \KK[\XX]/I = 3$, so $\ell - 1 + n = 4$. Here $(\mathfrak{m})^2$ denotes the ordinary square of the maximal ideal, not an ideal of the form $\mathfrak{m}^{\mathbf{b}}$; indeed, the irredundant irreducible decomposition of $(\mathfrak{m})^2 = (x_1^2, x_1 x_2, x_2^2)$ is $\mathfrak{m}^{(2,1)} \cap \mathfrak{m}^{(1,2)}$. More generally, for $a, b \geq 2$ the ideal $I = (x_1^a,\, x_1x_2,\, x_2^b) = \mathfrak{m}^{(a,1)} \cap \mathfrak{m}^{(1,b)}$ has $\ell = a + b - 1$, and its Demazure roots are $k\mathrm{e}_1$ for $1 \leq k \leq a-2$, $k\mathrm{e}_2$ for $1 \leq k \leq b-2$, each with $\dim \mathfrak{g}_\alpha = 1$, together with the two outer roots $(a-1)\mathrm{e}_1 - \mathrm{e}_2$ and $(b-1)\mathrm{e}_2 - \mathrm{e}_1$. Hence $\dim \Aut^0_\KK(A) = 2 + (a-2) + (b-2) + 2 = \ell - 1 + n$.
\end{remark}

\begin{lemma}\label{lemma: two variables}
Let $I \subset \KK[x_1,x_2]$ be a reducible monomial ideal such that $A = \KK[x_1,x_2]/I$ is finite-dimensional, and let $x_1^{a_1}, x_2^{a_2} \in G(I)$. Then $(a_1 - 1)\mathrm{e}_1 - \mathrm{e}_2$ and $(a_2 - 1)\mathrm{e}_2 - \mathrm{e}_1$ are outer Demazure roots of $I$. In particular, $s(I) \geq 2$.
\end{lemma}

\begin{proof}
We show that $\alpha = (a_1 - 1)\mathrm{e}_1 - \mathrm{e}_2$ belongs to $\mathcal{R}_{\mathrm{e}_2^*}(I)$; the other claim follows by symmetry. By criterion~\eqref{eq: outer criterion}, we must check that $\mathbf{c} + \alpha \in \supp(I)$ for every $\mathbf{c} \in G(I)$ with $c_2 \geq 1$. If $c_1 \geq 1$, then the first component of $\mathbf{c} + \alpha$ is $c_1 + a_1 - 1 \geq a_1$, so $\mathbf{c} + \alpha \in \supp(I)$. Otherwise $\mathbf{c} = a_2 \mathrm{e}_2$ and $\mathbf{c} + \alpha = (a_1 - 1, a_2 - 1)$. If this vector were not in $\supp(I)$, then every $\mathbf{m}$ with $m_i \leq a_i - 1$ would satisfy $\mathbf{m} \notin \supp(I)$, so that $I = \mathfrak{m}^{(a_1, a_2)}$ would be irreducible, a contradiction. Finally, $\mathrm{e}_2 + \alpha = (a_1 - 1)\mathrm{e}_1 \notin \supp(I)$, so the induced derivation is nonzero.
\end{proof}

\begin{lemma}\label{lemma: matching}
Let $\Gamma$ be a graph on $\{1, \ldots, n\}$ in which every vertex has degree at most $1$. Let $p$ be the number of edges of $\Gamma$, let $\sigma = n - 2p$ be the number of isolated vertices, and let $a_k \geq 2$ for each isolated vertex $k$. Consider the monomial ideal
$$
I = \bigl(x_k^{a_k} \mid k \text{ isolated}\bigr) + \bigl(x_i^2 \mid i \text{ matched}\bigr) + \bigl(x_ux_v \mid u \neq v,\ \{u,v\} \text{ not an edge}\bigr).
$$
Then $\ZZ^n_{\geq 0} \setminus \supp(I)$ consists exactly of $0$; $c\,\mathrm{e}_k$ with $1 \leq c \leq a_k - 1$ for $k$ isolated; and $\mathrm{e}_i$, $\mathrm{e}_j$, $\mathrm{e}_i + \mathrm{e}_j$ for each edge $\{i,j\}$. Moreover, $P(I) = p$ and $s(I) = n(\sigma + p - 1)$.
\end{lemma}

\begin{proof}
A monomial not in $I$ is divisible by at most two variables, and if it is divisible by $x_u$ and $x_v$ with $u \neq v$, then $\{u,v\}$ is an edge. Since $x_i^2 \in I$ for every matched vertex $i$ and $x_k^{a_k} \in I$ for every isolated vertex $k$, this gives the description of $\ZZ^n_{\geq 0} \setminus \supp(I)$. Among these elements, exactly the $p$ elements $\mathrm{e}_i + \mathrm{e}_j$ have $\nu = 2$, and the others have $\nu \leq 1$, so $P(I) = p$.

We now compute $s(I)$. By criterion~\eqref{eq: outer criterion}, the elements of $\mathcal{R}_{\mathrm{e}_j^*}(I)$ are the $\alpha = \mathbf{m} - \mathrm{e}_j$ with $\mathbf{m} \notin \supp(I)$, $m_j = 0$, and $\mathbf{m} + \mathbf{c} - \mathrm{e}_j \in \supp(I)$ for every $\mathbf{c} \in G(I)$ with $c_j \geq 1$.

If $j$ is isolated, the generators involving $x_j$ are $x_j^{a_j}$ and $x_jx_v$ for $v \neq j$, so the conditions read $\mathbf{m} + (a_j - 1)\mathrm{e}_j \in \supp(I)$ and $\mathbf{m} + \mathrm{e}_v \in \supp(I)$ for all $v \neq j$. The candidates $\mathbf{m} = 0$ and $\mathbf{m} = \mathrm{e}_i$ with $i$ matched fail, since $\mathrm{e}_v \notin \supp(I)$ and $\mathrm{e}_i + \mathrm{e}_{i'} \notin \supp(I)$ for the edge $\{i, i'\}$. The candidates $\mathbf{m} = c\,\mathrm{e}_k$ with $k \neq j$ isolated satisfy the conditions if and only if $c = a_k - 1$, and the candidates $\mathbf{m} = \mathrm{e}_i + \mathrm{e}_{i'}$ for an edge $\{i, i'\}$ always satisfy them. Hence $|\mathcal{R}_{\mathrm{e}_j^*}(I)| = (\sigma - 1) + p$.

If $j$ is matched with $j'$, the generators involving $x_j$ are $x_j^2$ and $x_jx_v$ for $v \notin \{j, j'\}$, so the conditions read $\mathbf{m} + \mathrm{e}_j \in \supp(I)$ and $\mathbf{m} + \mathrm{e}_v \in \supp(I)$ for all $v \notin \{j, j'\}$. The candidates $\mathbf{m} = 0$, $\mathbf{m} = \mathrm{e}_{j'}$ and $\mathbf{m} = \mathrm{e}_i$ with $i$ matched, $i \neq j'$, fail as before. The candidates $\mathbf{m} = c\,\mathrm{e}_k$ with $k$ isolated work if and only if $c = a_k - 1$, and the candidates $\mathbf{m} = \mathrm{e}_i + \mathrm{e}_{i'}$ for an edge $\{i,i'\} \neq \{j, j'\}$ always work. Hence $|\mathcal{R}_{\mathrm{e}_j^*}(I)| = \sigma + (p - 1)$.

Summing over the $\sigma$ isolated and $2p$ matched vertices, we get $s(I) = \sigma(\sigma - 1 + p) + 2p(\sigma + p - 1) = n(\sigma + p - 1)$.
\end{proof}

We now determine when the bound of \cref{proposition: cota} is attained.

\begin{theorem}\label{theorem: equality}
Let $I \subset \KK[\XX]$ be a reducible monomial ideal such that $A = \KK[\XX]/I$ is finite-dimensional of dimension $\ell$.
\begin{enumerate}[(i)]
\item $\dim \Aut^0_\KK(A) = \ell - 1 + n$ if and only if $n = 2$ and $I = (x_1^{a}, x_1x_2, x_2^{b})$ for some $a, b \geq 2$.
\item If $n \geq 3$, then $\dim \Aut^0_\KK(A) \geq \ell + n$. Moreover, for $n = 3$ equality holds for $I = (x_1^{a}, x_1x_2, x_1x_3, x_2^2, x_3^2)$ with $a \geq 2$.
\end{enumerate}
\end{theorem}

\begin{proof}
By \cref{lemma: identity},
\begin{equation}\label{eq: slack}
\dim \Aut^0_\KK(A) - (\ell - 1 + n) = s(I) + P(I) - n.
\end{equation}

\smallskip
\noindent\textit{The case $n = 2$.} Let $x_1^{a_1}, x_2^{a_2} \in G(I)$. By \cref{lemma: two variables}, $s(I) \geq 2$. Now \eqref{eq: slack} gives $\dim \Aut^0_\KK(A) - (\ell - 1 + n) \geq P(I) \geq 0$, with equality only if $P(I) = 0$, i.e., if $\nu(\mathbf{m}) \leq 1$ for every $\mathbf{m} \in \ZZ^2_{\geq 0} \setminus \supp(I)$, i.e., if $x_1 x_2 \in I$. In this case $\ZZ^2_{\geq 0} \setminus \supp(I)$ consists of $0$, $c\,\mathrm{e}_1$ with $1 \leq c < a_1$ and $c\,\mathrm{e}_2$ with $1 \leq c < a_2$, so $I = (x_1^{a_1}, x_1x_2, x_2^{a_2})$. Conversely, these ideals attain the bound by \cref{remark: sharp}. This proves (i) for $n = 2$.

\smallskip
\noindent\textit{The case $n \geq 3$.} Since $\dim \Aut^0_\KK(A) - (\ell - 1 + n)$ is a nonnegative integer by \cref{proposition: cota}, it suffices to show that it cannot vanish. Assume that $\dim \Aut^0_\KK(A) = \ell - 1 + n$, and set $\mathbf{1} = \sum_{i=1}^n \mathrm{e}_i$. We distinguish two cases.

\smallskip
\noindent\textit{Case A: $\mathbf{1} \in \widehat{\mathcal{M}}(I)$.} Then $\mathbf{1} \notin \supp(I)$, so $\mathbf{m} \notin \supp(I)$ for every $\mathbf{m} \in \{0,1\}^n$. The monomials $x_ix_j$ with $i < j$ and $x_ix_jx_k$ with $i < j < k$ contribute $\binom{n}{2} + 2\binom{n}{3} \geq n + 2$ to $P(I)$, so \eqref{eq: slack} gives $\dim \Aut^0_\KK(A) - (\ell - 1 + n) \geq 2$, a contradiction.

\smallskip
\noindent\textit{Case B: $\mathbf{1} \notin \widehat{\mathcal{M}}(I)$.} Since every summand on the right-hand side of \eqref{eq: slack decomposition} is nonnegative, our assumption forces $s = r'$, $\dim \mathfrak{g}_{\alpha_q} = 2$ for every $q$, and $\dim \mathfrak{g}_\alpha = 1$ for every $\alpha \in \Omega$. Since $\mathbf{1} \notin \widehat{\mathcal{M}}(I)$, every interior element of $\widehat{\mathcal{M}}(I)$ would fall in Case~1 of \cref{lemma: interior}, producing $\dim \mathfrak{g}_{\alpha_q} = n \geq 3$. Hence $\widehat{\mathcal{M}}(I)$ has no interior elements, and $|E_\alpha| = 1$ for every $\alpha \in (\ZZ^n_{\geq 0} \setminus \supp(I)) \setminus (\widehat{\mathcal{M}}(I) \cup \{0\})$. In particular, $|E_{\mathrm{e}_i}| \leq 1$ for every $1 \leq i \leq n$.

Let $\Gamma$ be the graph on $\{1, \ldots, n\}$ in which $i \neq j$ are adjacent if and only if $x_ix_j \notin I$. Since $|E_{\mathrm{e}_i}|$ equals the degree of $i$ in $\Gamma$ plus $1$ if $x_i^2 \notin I$, every vertex of $\Gamma$ has degree at most $1$, i.e., $\Gamma$ is a matching, and $x_i^2 \in I$ for every matched vertex $i$. Let $\mathbf{m} \in \ZZ^n_{\geq 0} \setminus \supp(I)$ with $\nu(\mathbf{m}) \geq 2$, and let $i \neq j$ with $m_i, m_j \geq 1$. Then $x_ix_j$ divides $\XX^{\mathbf{m}}$, so $i$ and $j$ are adjacent; if some $k \notin \{i,j\}$ satisfied $m_k \geq 1$, then $i$ would also be adjacent to $k$, which is impossible. If $m_i \geq 2$, then $2\mathrm{e}_i + \mathrm{e}_j \notin \supp(I)$, so $\mathrm{e}_i, \mathrm{e}_j \in E_{\mathrm{e}_i}$, again impossible. Hence $\mathbf{m} = \mathrm{e}_i + \mathrm{e}_j$. Let $p$ be the number of edges of $\Gamma$ and $\sigma = n - 2p$ the number of isolated vertices, and for each isolated vertex $k$ let $x_k^{a_k} \in G(I)$. Consequently, $I$ is the ideal of \cref{lemma: matching} attached to $\Gamma$.

By \cref{lemma: matching}, we have $P(I) = p$ and $s(I) = n(\sigma + p - 1)$, so \eqref{eq: slack} yields
$$
\dim \Aut^0_\KK(A) - (\ell - 1 + n) = n(\sigma + p - 1) + p - n = n(n - p - 2) + p,
$$
where we used $\sigma = n - 2p$. If $n = 3$, then $p \leq 1$ and this quantity equals $3$ or $1$. If $n \geq 4$, then $p \leq n/2$ gives $n - p - 2 \geq 0$, so the quantity is at least $p$, and it equals $n(n-2) \geq 8$ when $p = 0$. In all cases it is positive, contradicting our assumption. This proves (i) for $n \geq 3$ and the inequality in (ii).

\smallskip
Finally, the ideal $I = (x_1^{a}, x_1x_2, x_1x_3, x_2^2, x_3^2)$ is reducible, since $\mathcal{M}(I) = \{(a,1,1), (1,2,2)\}$, and it is the ideal of \cref{lemma: matching} attached to the graph with the single edge $\{2,3\}$ and the isolated vertex $1$, i.e., $n = 3$, $\sigma = 1$, $p = 1$. By \cref{lemma: matching} and \eqref{eq: slack}, we get $\dim \Aut^0_\KK(A) - (\ell - 1 + n) = 1$, while $\ell = 1 + (a - 1) + 3 = a + 3$.
\end{proof}

\bibliographystyle{alpha}
\bibliography{ref}
\end{document}